\documentclass[11pt]{article}
\usepackage[margin=1in]{geometry}
\usepackage{float}
\usepackage{graphicx}	
\usepackage{amsmath}
\usepackage{amsthm}
\usepackage{amssymb}
\newtheorem{theorem}{Theorem}
\newtheorem{lemma}{Lemma}

\begin{document}
\title{Exponential prime sequences}
\author{Bernard Montaron, Fraimwork SAS, France}
\date{November 1, 2020}
\maketitle
\begin{abstract}

Infinite exponential sequences of distinct prime numbers of the form $\lfloor a c^{n^d}+b\rfloor$, $n\geq 0$, are proved to exist for well chosen real constants $a>0$, $b$, $c>1$, $d>1$, assuming Cramer's conjecture on prime gaps. There is an infinity of such prime sequences.  Sequences having the least possible growth rate are of particular interest. This work's focus is on prime sequences with $a=1$, $b \in \{0,1\}$, that have the smallest possible constant $c$ given $d>1$, and sequences with the smallest possible $d$, given $c=2$. In particular, we prove the existence of the four infinite exponential prime sequences $u_0(n)=\lfloor c_0^{n\sqrt{n}}\rfloor$, $n\geq 1$, with $c_0=2.0073340803...$, $u_1(n)=1+\lfloor c_1^{n\sqrt{n}}\rfloor$, $n\geq 0$, with $c_1=2.2679962677...$, $v_0(n)=\lfloor 2^{n^{d_0}}\rfloor$, $n\geq 1$, with $d_0=1.5039285240...$, and $v_1(n)=1+\lfloor 2^{n^{d_1}}\rfloor$, $n\geq 0$, with $d_1=1.7355149500...$.

\end{abstract}
\section{Introduction}

In the famous sequence $u(n)=\lfloor c^{3^n}\rfloor$ for $n \geq 1$ introduced by Mills in 1947 (1) the smallest constant $c$ such that all $u(n)$ are prime numbers is equal to 1.306377883863... In this sequence the next term has three times more digits than the previous one, and that growth rate is the reason why only the first 11 terms are proven primes, and only the first 14 terms are probable primes.\\
The motivation of this paper is to generate infinite prime sequences with a much lower growth rate, low enough to generate hundreds of primes.\\
Let $f(x)$ be a non-negative growing function of $x$ such that $\forall x \geq x_0,\; f(x) \geq x$ for some $x_0>1$.  The infinite sequence $$u(n)=\lfloor a c^{f(n)}+b\rfloor$$ where $a>0$, $b$, $c>1$ are well-chosen real numbers is an exponential prime sequence if all terms for $n \in \mathbb{N}$ are distinct prime numbers. By convention if $b>0$ the sequence starts at $n=0$, otherwise it may start at $n=1$.\\
The condition $f(x) \geq x$ justifies the name exponential sequence. Functions such as, for example, $f(x)=\ln{(x^2 +x+41)}$ would not generate exponential sequences and will not be considered here.  Mills's sequence, which has $f(x)=3^x$, $a=1$, $b=0$, $c=1.306377883863...$, was the first exponential prime sequence introduced in mathematics.\\
Note that sequences of the form $\lceil a c^{f(n)} + b\rceil$ can also be considered as exponential prime sequences, but this paper will focus on forms with the floor function.\\
\\
Some proofs in this paper will be based on Cramer's conjecture (3) on prime gaps. It states that the gap $g(p)$ between a prime $p$ and the next prime is $O(\ln{^2 p})$ which implies $g(p) < H \ln{^2 p}$ for some constant $H$. Daniel Shanks subsequently conjectured (4) that $H=1$ for any prime greater than 7, but Andrew Granville (5) argued based on prime probability heuristics that $H > 2/\exp{\gamma}=1.1229...$ (6) where $\gamma$ is the Euler-Mascheroni constant. As of 2018, the greatest known ratio $g(p)/\ln{^2 p}$ for $p>7$ is 0.9206385886 with gap 1132 and $p=$1693182318746371.\\
The definition of the function $g(x)$ can be extended to any positive real $x$ as $g(x)=\text{pr}(x)-x$ where the function $\text{pr}(x)$ is the smallest prime number greater than $x$. Cramer's conjecture implies that for any $x>1$ real $g(x)<H\ln{^2 x}$.

\section{Exponential prime sequences with $f(x)=x^d$, $d>1$}

Let's consider the sequence $u(n)=\lfloor c^{n^d}\rfloor$ for $n \geq 1$ with an exponent $d$ greater than 1. Similar to Mills sequence, we're looking for the smallest constant $c$ such that all $u(n)$ are primes. Since $u(1)=\lfloor c\rfloor$ the constant $c$ must be greater or equal to 2. The initial value of the constant is set to $c_1=2$ and the sequence starts with the smallest prime $u(1)=c_1^{1^d}=2$.  Suppose that for some sufficiently large $n>1$ we were able to build a sequence of $n$ prime numbers $u(k)=\lfloor c_n^{k^d}\rfloor$ for all $0 \leq k \leq n$, with the constant equal to $c_n= u(n)^{1/n^d}$, i.e. $u(n)=c_n^{n^d}$ is an integer.\\  
Let's calculate the smallest possible following term.  If the number $U_n=\lfloor c_n^{(n+1)^d}\rfloor$ is prime, then the next term in the sequence is $u(n+1)=U_n$ and the constant $c_n$ does not need to be adjusted, i.e. we set $c_{n+1}=c_n$. On the contrary, if $U_n$ is not prime the constant $c_n$ must be increased to $c_{n+1}$ such that $u(n+1)=c_{n+1}^{(n+1)^d}$ is the smallest prime greater than $U_n$. At that stage $c_{n+1}=u(n+1)^{1/(n+1)^d}$ and $u(n+1)=c_{n+1}^{(n+1)^d}$ is an integer. The gap between $c_n^{(n+1)^d}$ and the next prime is $$g(U_n)=c_{n+1}^{(n+1)^d}-c_n^{(n+1)^d}$$
In order to obtain a prime sequence with this additional term a sufficient condition is that the new value $c_{n+1}$ of the constant doesn't change $u(n)$ and all previous terms. Note that this is not a necessary condition because, for example, if the new value of $u(n)$ is increased to another prime value, that's a valid case.

\begin{lemma}

Assuming that the first $n$ terms, $n>1$, of the sequence $u(k)=\lfloor c_n^{k^d} \rfloor$ are prime numbers and the constant is $c_n=u(n)^{1/n^d}$, let $\delta_k = c_n^{k^d}- \lfloor c_n^{k^d} \rfloor$ be the fractional part of the $n$ numbers $c_n^{k^d}$ for $k \leq n$. By assumption $\delta_n=0$. A sufficient condition for the next term $u(n+1)$ to be prime using a new constant $c_{n+1}\geq c_n$ that doesn't change the value of $u(n)$ is $$\delta_n^{(1)}=\frac{g(U_n)}{c_n^{(n+1)^d-n^d}}\big(\frac{n}{n+1}\big)^d+O\Big(\frac{g(U_n)^2}{(U_n+\beta_n)^{2-(n/(n+1))^d}}\Big)<1$$ where $\delta_n^{(1)}=c_{n+1}^{n^d}-\lfloor c_{n+1}^{n^d}\rfloor$ is the fractional part associated to $u(n)$ when using the new constant value $c_{n+1}$, and where $\beta_n$ is the fractional part of $c_n^{(n+1)^d}$.

\end{lemma}
\begin{proof}

The value of $u(n)$ is unchanged iff $0 \leq \delta_n^{(1)} =c_{n+1}^{n^d}-c_n^{n^d}<1$.  By construction $c_n^{n^d}$ is an integer hence $\delta_n^{(1)}$ is equal to the fractional part $c_{n+1}^{n^d}-\lfloor c_{n+1}^{n^d}\rfloor$. Let's calculate $\delta_n^{(1)}$.
$$c_{n+1}^{n^d}=c_n^{n^d}\Big(1+\frac{g(U_n)}{c_n^{(n+1)^d}}\Big)^{(n/(n+1))^d}=c_n^{n^d}\Big(1+ \frac{g(U_n)}{U_n+\beta_n}\big(\frac{n}{n+1}\big)^d+R\Big)$$
with $0<\beta_n =c_n^{(n+1)^d}-U_n<1$ and where the remainder is $R=O(g(U_n)^2/(U_n+\beta_n)^2)$. With this, $$\delta_n^{(1)}=U_n^{(n/(n+1))^d}\Big(\frac{g(U_n)}{U_n+\beta_n}\big(\frac{n}{n+1}\big)^d +R\Big)=\frac{g(U_n)}{(U_n+\beta_n)^{1-(n/(n+1))^d}}\big(\frac{n}{n+1}\big)^d +O\Big(\frac{g(U_n)^2}{(U_n+\beta_n)^{2-(n/(n+1))^d}}\Big)$$
and from $U_n+\beta_n=c_n^{(n+1)^d}$, $$\delta_n^{(1)} =\frac{g(U_n)}{c_n^{(n+1)^d-n^d}}\big(\frac{n}{n+1}\big)^d+O\Big(\frac{g(U_n)^2}{(U_n+\beta_n)^{2-(n/(n+1))^d}}\Big)<1$$.

\end{proof}

The prime number theorem implies that the mean value of $g(U_n)$ is on the order of $\ln{U_n}$. Therefore, in the case $d=1$, the condition above implies that it would be impossible \footnote[1]{Following the work of Yitang Zhang and the Polymath project set up by Terence Tao (2), it is known that there are infinitely many prime gaps less than 246. But since the mean prime gap around $U_n$ is $\ln{U_n}$, the probability for $g(U_n)\leq 246$ gets lower and lower as $U_n$ increases, and finding an infinite sequence $u(n)$ such that almost all gaps are less than 246 has probability 0. So it is safe here to conjecture that this cannot happen.} to build an infinite sequence of primes $u(n)$ because the constant $c$ would have to be gradually increased to be kept above $g(U_n)\approx \ln{U_n}$ and can never converge. That is why the exponent $d$ is taken greater than 1.\\
\\
It is important to note here that since Cramer's conjecture on prime gaps implies  $g(U_n) <H\ln{^2 U_n}$ for some $H$ on the order of 1, and assuming that $U_n$ is large enough, for example on the order of $U_n\approx 10^{100}$, then $g(U_n)^2/(U_n+\beta_n)^{2-(n/(n+1))^d}<g(U_n)^2/(U_n+\beta_n)<\ln{^2 U_n}/(U_n+\beta_n)\approx 2.8\; 10^{-91}$. The big O term in the lemma  1 condition can be ignored provided $U_n$ is sufficiently large. Given $\epsilon>0$ as small as desired the condition in lemma 1 can be written $$\delta_n^{(1)}=\frac{g(U_n)}{c_n^{(n+1)^d-n^d}}\big(\frac{n}{n+1}\big)^d < 1-\epsilon$$ with $U_n$ large enough so that $\epsilon$ is greater than the remainder term.\\
\\
Assuming Cramer's conjecture, the sufficient condition in lemma 1 can be written in the form $c^{(n+1)^d-n^d}/n^d (n+1)^d > (\ln{c})^2$ which gives a lower bound on the minimum number $n$ of consecutive prime terms that must be built for the beginning of the sequence. Above that number of prime terms, the conditions of lemma 1 and lemma 2 below are achieved. For example, with $d=1.5$ and $c =2.268$ this condition gives $n \geq 136$.  With $d=1.2$ and with $c \leq 20$ the condition is $n \geq 19351$, i.e. way out of practical range, but this is a very restrictive condition...\\
\\
Lemma 1 provides a condition for adding one prime term to an existing sequence of $n$ terms without changing $u(n)$. The same approach can be used to obtain the condition for none of the first $n$ terms to be changed, this is the 

\begin{lemma}

If all $n$ first terms, $n>1$, of the sequence $u(k)=\lfloor c_n^{k^d} \rfloor$ are prime numbers and the constant is $c_n=u(n)^{1/n^d}$, a sufficient condition given $n-1$ numbers $0<\epsilon_k<1$, $1\leq k <n$, for adding the next prime term $u(n+1)$ immediately above $c_n^{(n+1)^d}$ using a new constant $c_{n+1}=u(n+1)^{1/(n+1)^d}> c_n$ that doesn't change the value of any of the first $n$ terms is $$\forall k \leq n-1,\; \; \delta_k < 1-\epsilon_k - \delta_n^{(1)} \frac{(k/n)^d}{c_n^{n^d-k^d}}$$
where $\delta_k=c_n^{k^d}-\lfloor c_n^{k^d} \rfloor$ and $\delta_n^{(1)}<1$ is given by lemma 1, and where $u(n)$ is sufficiently large for the remainder terms to be less than $0<\epsilon_k <1$.

\end{lemma}
\begin{proof}

Let the $\delta_k$ be the fractional parts of $c_n^{k^d}$ for $k\leq n$, the approach used in the proof of lemma 1 immediately leads to the sufficient condition $$\forall k \leq n, \;\;\delta_k + \frac{g(U_n)}{c_n^{(n+1)^d-k^d}}\big(\frac{k}{n+1} \big)^d < 1-\epsilon_k $$ with remainder terms smaller than $\epsilon_k$ provided $u(n)$ is sufficiently large.\\ 
Since $\delta_n=0$ the condition on $\delta_n^{(1)}$ is simply $\delta_n^{(1)} <1-\epsilon_n$, as in lemma 1.  This implies that $\delta_k^{(1)} = \delta_n^{(1)} (k/n)^d/c_n^{n^d-k^d}$ and the conditions become $$\forall k \leq n-1,\; \; \delta_k < 1-\epsilon_k - \delta_n^{(1)} \frac{(k/n)^d}{c_n^{n^d-k^d}}$$ 

\end{proof}

Lemma 2 gives a sufficient condition for keeping unchanged the values of the first $n$ terms in an exponential prime sequence of the form $u(n)=\lfloor c^{n^d} \rfloor$ when adding a new prime term $u(n+1)$ that requires increasing slightly the constant $c$. The condition can be written in the form $$\forall k \leq n-1,\; \; \delta_k^{(1)} =\delta_k + \frac{g(U_n)}{c_n^{n^d-k^d}}\Big(\frac{k}{n}\Big)^d< 1-\epsilon_{1k}$$ where the $\delta_k^{(1)}=c_{n+1}^{k^d}-\lfloor c_{n+1}^{k^d} \rfloor$ are the new fractional parts recalculated for the first $n-1$ terms with the new constant $c_{n+1}$. Let's now look at the effect on the $\delta$ fractional parts of adding one more prime term $u(n+2)$ to the sequence.\\
\\
As seen above, $c_{n+1}=u(n+1)^{1/(n+1)^d}$ and $\delta_{n+1}=0$. If $\lfloor c_{n+1}^{(n+2)^d} \rfloor$ is prime there is no change on any of the fractional parts of the first $n+1$ terms of the sequence and the constant remains $c_{n+1}$. In the case where $\lfloor c_{n+1}^{(n+2)^d} \rfloor$ is not prime, we pick the first prime following that number and assign that value to $u(n+2)$. The new constant is $c_{n+2}=u(n+2)^{1/(n+2)^d} > c_{n+1}$. Applying lemma 2 to this sequence of $n+1$ prime terms (instead of $n$ terms) we immediately obtain $$\forall k \leq n,\; \; \delta_k^{(2)} =\delta_k^{(1)} + \frac{g(U_{n+1})}{c_{n+1}^{(n+1)^d-k^d}}\Big(\frac{k}{n+1}\Big)^d< 1-\epsilon_{2k}$$  where $\delta_k^{(2)}$ are the fractional parts of the $c_{n+2}^{k^d}$ for $k\leq n+1$. Using the formula for $\delta_k^{(1)}$ in lemma 2 we obtain $$\forall k \leq n,\; \; \delta_k^{(2)} =\delta_k +\frac{g(U_n)}{c_n^{n^d-k^d}}\Big(\frac{k}{n}\Big)^d+  \frac{g(U_{n+1})}{c_{n+1}^{(n+1)^d-k^d}}\Big(\frac{k}{n+1}\Big)^d < 1-\epsilon_{2k}$$
Note that this is the worst case where $\lfloor c_n^{(n+1)^d} \rfloor$ and $\lfloor c_{n+1}^{(n+2)^d} \rfloor$ are not prime, which requires to increase the constant $c$ for these two additional terms. That generates the greatest possible increase of the fractional parts $\delta_k^{(2)}$. Adding more terms, up to $u(n+j)$, and again assuming the worst case where at each step the constant must be increased, we get $$\forall k \leq n+j-2,\; \; \delta_k^{(j)} =\delta_k +\frac{g(U_n)}{c_n^{n^d-k^d}}\Big(\frac{k}{n}\Big)^d +...+ \frac{g(U_{n+j-1})}{c_{n+j-1}^{(n+j-1)^d-k^d}}\Big(\frac{k}{n+j-1}\Big)^d < 1-\epsilon_{jk}$$
where the respective $0<\epsilon_{jk}<1$ are greater than the sum of all second order terms, and made very small by choosing $n$ and $U_n$ large enough. We have proved the

\begin{lemma}

If $j$ new prime terms are added to a sequence like in lemma 2 of $n$ prime terms $u(k)$, $1\leq k \leq n$, and in the worst case where the constant must be increased at each step i.e. $c_n<c_{n+1}<...<c_{n+j}$, the fractional parts of the first $n$ terms initially equal to $\delta_k=c_n^{k^d}-\lfloor c_n^{k^d}\rfloor$ become $\delta_k^{(j)}$ where
$$\forall k \leq n,\; \; \delta_k^{(j)} =\delta_k +\frac{g(U_n)}{c_n^{n^d-k^d}}\Big(\frac{k}{n}\Big)^d +...+ \frac{g(U_{n+j-1})}{c_{n+j-1}^{(n+j-1)^d-k^d}}\Big(\frac{k}{n+j-1}\Big)^d < 1-\epsilon_{jk}$$
where the respective $0<\epsilon_{jk}<1$ are greater than the sum of all second order terms, and made very small by choosing $n$ and $U_n$ large enough.

\end{lemma}

The best way to understand why these three lemmas can explain how an infinite exponential prime sequence can be built, is to actually build one!

\section{Algorithm for building exponential prime sequences}

The following algorithm (A) is applied to build a sequence $u(n)=\lfloor ac^{n^d}+b \rfloor$ for some fixed $a$, $b$, $d>1$. Let $a=1$, $b=0$, and the sequence starts at $n=1$ with the smallest prime number $u(1)=2$ which implies $c\geq 2$. The constant is initialized at the minimum value $c=2$.\\
\\
0.   $n=1$, $u(1)=2$, $c=2$, $a=1$, $b=0$\\
1.   $n=n+1$\\
2.   $q=\lfloor ac^{n^d}+b\rfloor$\\
3.   $p=\text{smpr}(q)$\\
4.   If $p=q$ then $u(n)=p$, go to 1.\\
5.   $c=\big((p-b)/a\big)^{1/n^d}$\\
6.   $u(n)=p$\\
7.   $k=1$\\
8.   $q=\lfloor ac^{k^d+b} \rfloor$\\
9.   If $q \neq u(k)$ and q not prime then $p=\text{smpr}(q)$, $n=k$, go to 5.\\
10. If q is prime then u(k)=q\\
11. If $k<n-1$ then $k=k+1$, go to 8.\\
12. go to 1.\\
\\
The function $\text{smpr}(q)$ returns the smallest prime greater or equal to $q$.\\
By design, this algorithm builds the increasing sequence of prime numbers of the form $\lfloor ac^{n^d}+b\rfloor$ that has the smallest possible growth rate. At each step the smallest prime number is selected. Hence, given $a$, $b$, $d$, algorithm (A) generates an exponential sequence with the smallest possible constant $c$.\\
Computations must be done at full precision, i.e. with a precision sufficient to ensure that the first 5 decimals of $c^{n^d}$ are correct, whatever the size of the integral part.\\
Let's apply the algorithm with $d=1.5$.  Below is the list of the main steps up to reaching numbers with 100 digits.
\\
$u(1)=2$, $c=2$\\ 
$u(2)=7 \; \; \rightarrow \; \; u(3)=37$, $c=37^{1/3\sqrt{3}}$\\
The symbol $\rightarrow$ indicates that the same constant $c$ is used for the next term, otherwise a new value of $c$ is calculated from each new term $u(n)$. Among the first 50 terms of this sequence the $\rightarrow$ happens only for $u(3)$.\\ 
$u(4)=263$, $c=263^{1/4\sqrt{4}}$\\ 
$u(5)=2411$, $c=2411^{1/5\sqrt{5}}$\\ 
$u(6)=27941$, $c=27941^{1/6\sqrt{6}}$\\
$u(7)=400823$, $c=400823^{1/7\sqrt{7}}$\\
\\
Back to $u(5)=2417$\\ 
``Back to'' means that the test $q \neq u(k)?$ at step 9 in the algorithm is positive and the sequence must be re-calculated from index $k$.\\
$u(6)=28019$\\ 
$u(7)=402197$\\ 
$u(8)=7035797$\\
\\
Back to $u(7)=402221$\\
\\
Back to $u(6)=28027$\\ 
$u(7)=402341$\\ 
$u(8)=7038841$\\ 
$u(9)=148159931$\\ 
$u(10)=3712563089$\\ 
$u(11)=109757564149$\\ 
$u(12)=3798944559521$\\ 
$u(13)=152911104014639$\\ 
$u(14)=7115199903124633$\\ 
$u(15)=380726791842607163$\\ 
$u(16)=23316218510440590877$\\ 
$u(17)=1627263370489435571639$\\ 
$u(18)=128918882760677617873919$\\
\\
Back to $u(17)=1627263370489435571681$\\ 
$u(18)=$128918882760677617877491\\ 
$u(19)=$11552650889301910967243591\\ 
$u(20)=$1167145691555195174453680873\\ 
$u(21)=$132533994626692147587330006821\\ 
$u(22)=$16868027334991412027934269842393\\ 
$u(23)=$2399919102359943086759656808542601\\ 
$u(24)=$380767545057477655978563593953193167\\
$u(25)=$67213615800749068530221664858688432411\\ 
$u(26)=$13172010994867150677120721027635155958683\\ 
$u(27)=$2859976570693485153112034466187708219743257\\
$u(28)=$686682528617038617430680064268466968498081537\\ 
$u(29)=$181989044798183541663688889370297517515196111517\\
$u(30)=$53147632265048501348397745176098125110334396704491\\
$u(31)=$17075104848461039007372851133376250502663750204947827\\ 
$u(32)=$6025728503245542099163202230622130655546040081475114503\\
$u(33)=$2332277743697360397390383770740373981929559594283584464517\\
$u(34)=$988695902920484127543775248274484931244783811540124044143033\\
$u(35)=$458428319483198619698559581351272118153426311552062470034569911\\
$u(36)=$232191407352013579551444284181079064953131941642986951678503259733\\
$u(37)=$128306935177570479232920344024389038414863778500648075219839796428039\\
$u(38)=$77262471513030250393502853824944255913200135566618300105631736306980607\\
$u(39)=$50641574378922399251192330014656382967577822893810504887594849344704551703\\
$u(40)=$36090320621425216445497140034432289541679239492008165225937294068619639467327\\
$u(41)=$27935854540311568015539342808807285187849565278956205286384010899263850602321327\\
$u(42)=$2346283232850058476656838232717336547807943004886017790987758154112398972036101\\6661\\
$u(43)=$2136099973721157678533842243132596366610288351890890567834808941420975605705163\\5953587\\
$u(44)=$2106083963954103006339535124533411418548377237436001619414293044282861289173777\\9289631637\\
$u(45)=$2246710501681377754448511625819629174477838687607396877814751789048416221385980\\3683513534369\\
$u(46)=$2590916996541251525250668114899009574972660678009900231450290517450857361774622\\8489871748546723\\
$u(47)=$3227193447206512086165538429323702740932449404490509579108859474670239870494237\\6733125054278699441\\
\\
$u(47)=c_{47}^{47^{1.5}}$ has 98 digits and $U_{47}=c_{47}^{48^{1.5}}$ has 101 digits. The largest gaps $g(U_n)$ relative to $\ln{U_n}$ are for $u(31)$ and $u(34)$ and are less than 2.5 times $\ln{U_n}$. That shows that, indeed, gaps close to Cramer's bound are rare.\\
The fractional parts at this stage $\delta_k=c_{47}^{k^{1.5}}-\lfloor c_{47}^{k^{1.5}} \rfloor$ for $k \leq 47$ are listed in Fig.1.\\

\begin{figure}[H]
\centering
\includegraphics[width=16cm]{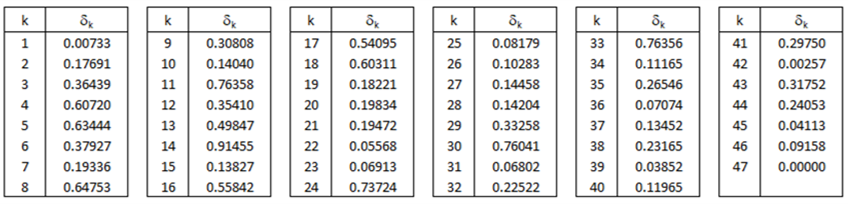}
\caption{Fractional parts of the first 47 terms of the sequence.}
\end{figure}

Assuming Cramer's conjecture on prime gaps, the remainder term $g(U_{47})^2/(U_{47}+\beta)^{2-(47/48)^{1.5}}$ corresponding to $\delta_{47}$ is less then than $H^2(\ln{U_{47}})^4/U_{47}^{2-(47/48)^{1.5}}=H^2 \times 4.9\;10^{-95}$ and remainder terms for $\delta_{k}$, $k<47$, are even smaller. The actual gap is $g(U_{47})=282.09792$, taking $\epsilon=10^{-90}$, the conditions in lemma 2 are applied with $\delta_{47}^{(1)}=g(U_{47})/c_{47}^{48^{1.5} - 47^{1.5}}(47/48)^{1.5}=0.20333$.\\
The condition for $\delta_{46}=0.09158$ is $0.09158 <1-10^{-90}-0.20333\times 0.000777\approx1-0.000158$.  Even in the unlikely case where $g(U_{47})$ would reach Cramer's bound, i.e. $g(U_{47})=H(\ln{U_{47}})^2=53696.9H$,with $H$ on the order of 1, this condition would be easily met. The condition for $\delta_{45}=0.04113$ is $0.04113<1-10^{-90}-0.20333\times 6.5\;10^{-7}\approx1-1.326\;10^{-7}$ which is even easier to meet. The next condition for $\delta_{44}=0.24053$ is $0.24053<1-1.2\;10^{-10}$. And the lower bound on the remaining $1-\delta_k$ for $k<44$ is reduced by a factor exceeding 1000 at each new index, so clearly the condition is easily met for $n=47$. Adding a new prime term $u(48)$ to the sequence with the corresponding new constant $c$ cannot change the values of the first 47 terms.\\
\\
The following terms of this sequence are:\\
$u(48)=$4338140817630590214054362195048689886488509155467665568696499290296446990848939\\8427666331327222715203\\
$u(49)=$6288437070446358179496545645990968999812362882489779253350947128112646599131068\\9433673932756599509788977\\
$u(50)=$9822141176049090549430695603783685455105957169682029891558596067419927018354320\\0955293790927131718028816721\\
$u(51)=$1651840421326810217697918871887828016986002578950753669357369582817169506165749\\05043497577386969941967834911183\\
$u(52)=$2988905549043481124877205279918506435106071579929536761949545055372561077203559\\34420944042180902829312816977339237\\
$u(53)=$5814754291308033646843039713558479910118847409036100671880024834313735891706342\\42894339197382912255444770492516788529\\
$u(54)=$1215422420362089857501177287341544041537856372600749118270090298422372468147065\\349680471669738540241709617786322609697869\\
$u(55)=$2727782697857489146106788892659849021703482123147301564886930368818656714891925\\222491228227878201405208218605861730056224351\\
\\
$u(55)$ has 124 digits and it requires 141 decimals (rounded) accuracy on $c$, i.e. \\
$c=$2.007334080333576928956181316653840772831434461478017370705960293321699050171486297\\
563613802880600998767775141388245083072913544366923411581602

\section{Existence of infinite exponential prime sequences}

Algorithm (A) is used to calculate the first $n$ prime terms of the sequence $u(k)=\lfloor c^{k^d}\rfloor$ with $d>1$ and where $c_n=u(n)^{1/n^d}$. We continue applying algorithm (A) to generate the next $j$ terms. By construction $c_n\leq c_{n+1}\leq...\leq c_{n+j-2}\leq c_{n+j}$.  Given the first $n$ terms, to prove the existence of an infinite prime sequence $u(n+j)=\lfloor c^{(n+j)^d}\rfloor$ for $j\geq 1$ it is sufficient to prove that $c_{n+j}$ is bounded. This is the

\begin{theorem}

Given the first $n$ terms of the sequence $u(k)=\lfloor c_n^{k^d}\rfloor$, $d>1$, computed with algorithm (A), where $c_n=u(n)^{1/n^d}$, and with $(n+1)^d$ sufficiently large, and assuming Cramer's conjecture on prime gaps $g(p)<H\ln{^2 p}$ for any prime $p>7$, the increasing sequence ${c_{n+j}}$ for $j\geq 1$ is bounded by
$$c_{n+j} < c_n + 2H\ln{^2 2}\sum_{k=1}^{j}\frac{(n+k)^d}{2^{(n+k)^d}}$$
therefore it converges to a constant $c$ that is the smallest constant such that all terms of the infinite sequence ${\lfloor c^{k^d}\rfloor}$ for $k \geq 1$ are distinct prime numbers.

\end{theorem}
\begin{proof}

If $\lfloor c_n^{(n+1)^d}\rfloor$ is prime then the constant is unchanged, i.e. $c_{n+1}=c_n$. And if $\lfloor c_n^{(n+1)^d}\rfloor$ is not prime, the algorithm picks for the prime number immediately above $\lfloor c_n^{(n+1)^d}\rfloor$ and assigns it to $u(n+1)$. The constant becomes $c_{n+1}=u(n+1)^{1/(n+1)^d}$. Cramer's conjecture implies that 
$$u(n+1)-c_n^{(n+1)^d}<H\ln{^2 c_n^{(n+1)^d}}=H(n+1)^{2d}\ln{^2 c_n}$$
hence $$c_{n+1}<c_n\Big(1+\frac{H(n+1)^{2d}\ln{^2 c_n}}{c_n^{(n+1)^d}}\Big)^{1/(n+1)^d}<c_n+\frac{H(n+1)^d c_n\ln{^2 c_n}}{c_n^{(n+1)^d}}$$
For the next term $u(n+2)$ the same approach leads to $$c_{n+2}<c_{n+1}+\frac{H(n+2)^d c_{n+1}\ln{^2 c_{n+1}}}{c_{n+1}^{(n+2)^d}}\leq c_{n+1}+\frac{H(n+2)^d c_n\ln{^2 c_n}}{c_n^{(n+2)^d}}$$
since $c_{n+1} \geq c_n \geq 2$, and for $(n+1)^d \geq 1+2/\ln{2}\geq 1+2/\ln{c}$ and $c \geq 2$ the function $F(c)=c\ln{^2 c}/c^{(n+1)^d}$ is decreasing. Combining the two last results $$c_{n+1} < c_n + \frac{H(n+1)^d c_n\ln{^2 c_n}}{c_n^{(n+1)^d}} + \frac{H(n+2)^d c_n\ln{^2 c_n}}{c_n^{(n+2)^d}}$$
and extending this to $n+j$ $$c_{n+j} < c_n + Hc_n\ln{^2 c_n}\sum_{k=1}^{j}\frac{(n+k)^d}{c_n^{(n+k)^d}} \leq c_n + 2H\ln{^2 2}\sum_{k=1}^{j}\frac{(n+k)^d}{2^{(n+k)^d}}$$
provided $(n+1)^d\geq 1+2/\ln{2}=3.88539...$.

\end{proof}

Algorithm (A) successfully generated the first 47 terms of the exponential prime sequence $u(k)=\lfloor c_{47}^{k^{1.5}}\rfloor$ with $c_{47}=u(47)^{1/47^{1.5}}=$2.00733408... and $(n+1)^d=48^{1.5}=332.55...$ therefore, as per Theorem 1, an infinite sequence of primes exists with a constant that converges to $c$ with $c_{47}<c<c_{47}+\epsilon$ where $$\epsilon = 2H\ln{^2 2}\sum_{k\geq 1}(47+k)^{1.5}/2^{(47+k)^{1.5}}=H\times 2.490052430529\;10^{-98}$$
Even if we take a conservative $H=10$ the very crude bound of Theorem 1 remains strong.  In fact, the first 100 decimals of $c_{47}$ are correct, i.e. identical to the decimals of the limit constant $c$.\\
\\ 
Theorem 1 can be extended immediately to sequences with $a=1$ and $b=1$ starting at $n=0$.\\
It is possible to generalize Theorem 1 to sequences with $a \neq 1$ and $b \neq 0$ following the same approach, i.e. given the first $n$ terms of the exponential prime sequence $u(k)=\lfloor ac_n^{k^d}+b\rfloor$ with $1 \leq k \leq n$, $a>0$, $b\geq 0$, $c_n\geq 2$, $d>1$, where $c_n=\big((u(n)-b)/a\big)^{1/n^d}$, and with $(n+1)^d$ sufficiently large, and assuming Cramer's conjecture on prime gaps $g(p)<H\ln{^2 p}$ for any prime $p>7$, the increasing sequence ${c_{n+j}}$ for $j\geq 1$ converges to a limit $c$ close to $c_n$ leading to an infinite exponential prime sequence of the form $\lfloor ac^{k^d}+b\rfloor$ for $k \geq 1$. However there are some technicalities depending on the values of $a$ and $b$, and depending on the starting index of the sequence at $n=0$ or $n=1$.\\

\begin{theorem}

The infinite exponential prime sequences $u_0(k)=\lfloor c_0^{k\sqrt{k}}\rfloor$, $k\geq 1$, $u_1(k)=1+\lfloor c_1^{k\sqrt{k}}\rfloor$, $k\geq 0$, $v_0(k)=\lfloor 2^{k^{d_0}}\rfloor$, $k\geq 1$, and $v_1(k)=1+\lfloor 2^{k^{d_1}}\rfloor$, $k\geq 0$ exist, and the smallest possible values of their respective constants are $$c_0=2.007334080333576928956181316653... \; \; \; \; c_1=2.267996267706724247328553280725...$$ $$d_0=1.503928524069520633527689067897... \; \; \; \;    d_1=1.73551495003300118213908371144...$$

\end{theorem}
\begin{proof}

This was already established above for the sequence ${u_0(k)}$. The algorithm (A) is applied in the next section with $b=1$ to generate the first 50 terms of the sequence ${u_1(k)}$ which proves its existence, as per Theorem 2. For the sequence ${v_0(k)}$, algorithm (A) is changed to algorithm (B) below\\
\\
0.   $n=1$, $v(1)=2$, $a=1$, $b=0$, $d=1$\\
1.   $n=n+1$\\
2.   $q=\lfloor a 2^{n^d}+b\rfloor$\\
3.   $p=\text{smpr}(q)$ smallest prime $\geq q$\\
4.   If $p=q$ then $v(n)=p$, go to 1.\\
5.   $d=\ln{(\ln{\big((p-b)/a\big)}/\ln{2})}/\ln{n}$\\
6.   $v(n)=p$\\
7.   $k=1$\\
8.   $q=\lfloor a 2^{k^d+b} \rfloor$\\
9.   If $q \neq v(k)$ and q nor prime then $p=\text{smpr}(q)$, $n=k$, go to 5.\\
10. If q is prime then u(k)=q\\
11. If $k<n-1$ then $k=k+1$, go to 8.\\
12. go to 1.\\
\\
Following the same approach as above, similar results are obtained for the convergence of the constant $d$. The first 50 terms of the sequence  ${v_0(k)}$ and the first 40 terms of  ${v_1(k)}$ are listed in the next section, and that is sufficient to establish the proof of existence of the corresponding infinite exponential prime sequences.

\end{proof}

\section{Sequences ${u_1(k)}$, ${v_0(k)}$, and ${v_1(k)}$}

The smallest constant of the sequence of primes $u_1(k)=1+\lfloor c_1^{k\sqrt{k}}\rfloor$, $k\geq 0$, is\\
\\
$c=$ 2.2679962677067242473285532807253717745270422544008187722755908290507837407514695735\\
72173836290992570042731587317115765881934097281139085997966167707719867142706558816816\\
\\
The first 50 terms of this sequence are\\
$u_1(0)=$2, $u_1(1)=$3, $u_1(2)=$11, $u_1(3)=$71, $u_1(4)=$701, $u_1(5)=$9467, $u_1(6)=$168599, $u_1(7)=$3860009, $u_1(8)=$111498091, $u_1(9)=$4002608003, $u_1(10)=$176359202639, $u_1(11)=$9437436701437,\\
$u_1(12)=$607818993573569, $u_1(13)=$46744099128452807, $u_1(14)=$4262700354254812091,\\ 
$u_1(15)=$458091929703695291747\\ 
$u_1(16)=$57691186909930154615407\\ 
$u_1(17)=$8471601990692484416847631\\ 
$u_1(18)=$1443868262009075144775972529\\ 
$u_1(19)=284427125290802666440635500171$\\ 
$u_1(20)=$64508585190035762546819414183627\\ 
$u_1(21)=$16784808124437614932174625219743493\\ 
$u_1(22)=$4993776940608054381884859117354140207\\ 
$u_1(23)=$1693622624439192850596898688425684156253\\ 
$u_1(24)=$652872150303652679829448688982189886530659\\
$u_1(25)=$285293542323304888886344519771469937143536487\\
$u_1(26)=$140963744324927916477364026537883297114648288919\\
$u_1(27)=$78566634770350201851112613233249621352518449931451\\
$u_1(28)=$49283986294561627960640555061706160769690906854022923\\
$u_1(29)=$34720429406470258768183815914888303901101991774709965431\\
$u_1(30)=$27415670531671050330126039277336819265920565545391648119999\\
$u_1(31)=$24216667131296755001605322259453346627082514305800550720897993\\
$u_1(32)=$23885779872975194583233136691346966939859444064807445569345585769\\
$u_1(33)=$26261487112714455828805540159203377478469231163597234973363591457207\\
$u_1(34)=$32131715178214076067056795236138505944078626999475422937130486192963489\\
$u_1(35)=$43681253619905074621450113773989825418769170136321252379454614766606166359\\
$u_1(36)=$65878617212511188886784609388021172886969769867239371275593094704456564034821\\
$u_1(37)=$1100655205873376127672036488377342032903426476893608541350969641654217697637525\\29\\
$u_1(38)=$2034279120239700202948384500215003905339895300407103680071470099009178559685533\\51399\\
$u_1(39)=$4153758739423584412992693371228397619324026582077933149326239916352317421725935\\05853191\\
$u_1(40)=$9358042492569802729520609587374143828463125863469529904345234046340622933898926\\64419564529\\
$u_1(41)=$2323295217517751505932465313890362065365858794901067634239480870103584142721470\\881798817966351\\
$u_1(42)=$6348656563110313224340650976172178458119581049994369593576418852525530460001700\\896268380592778881\\
$u_1(43)=$1907297023897408775079723704632843129987170748455533261915142025781301242769869\\5059821888054480341117\\
$u_1(44)=$6292633153066628726657827171980706234316263190667552947350607625776104397190081\\4762388755632574676703581\\
$u_1(45)=$2277503850410776788204699205894554449275849264408643883808922616406190120942257\\48681562377621385333193418207\\
$u_1(46)=$9033354941221173708960209833371771048958695305824550656593453977805407161259864\\92500646589837326868534021348177\\
$u_1(47)=$3922542367498515739383907442359821276578045509186299075524781752470520738676108\\830582907866229737347691015877811937\\
$u_1(48)=$1862919706700012529715396178894400735117593576660314270680981225141288481855631\\4349681115112062044781842825756034795743\\
$u_1(49)=$9667653742163195831473831157507697480390514020459651370786988940935447881047722\\6127703149269413503506546288674947017660273\\
$u_1(50)=$5477136514912296033167409439922046558480310261848658631193969596882946672229755\\57349070305660634485723007668019082775049887609\\
\\
For more terms refer to OEIS: A338613. Many thanks to Francois Marques for extending the computation of this sequence to 250 terms, and for checking my calculations for other sequences.\\
\\
The smallest constant of the sequence of primes $v_0(k)=\lfloor 2^{k^{d_0}}\rfloor$, $k\geq 1$, is\\
\\
$d_0=$1.50392852406952063352768906789758319919073884958113842900299935065765954756163057\\
643171018908088652246874013053730909403941879035447355628459322003109068400114528356779\\
\\
The first 50 terms of this sequence are\\
$v_0(1)=$2, $v_0(2)=$7, $v_0(3)=$37, $v_0(4)=$263, $v_0(5)=$2437, $v_0(6)=$28541, $v_0(7)=$414893, $v_0(8)=$7368913, $v_0(9)=$157859813, $v_0(10)=$4035572951, $v_0(11)=$122006926709, $v_0(12)=$4328504865941,\\ 
$v_0(13)=$178988464493359, $v_0(14)=$8575347401843113, $v_0(15)=$473485756611713633,\\ 
$v_0(16)=$29985730185033339911, $v_0(17)=$2168685169398896331137,\\ 
$v_0(18)=$178419507110725228550743, $v_0(19)=$16637432012996855576590853,\\
$v_0(20)=$1752619810931482386016278601\\
$v_0(21)=$207929078572140606038877989927\\
$v_0(22)=$27703432292277234733469436505007\\
$v_0(23)=$4134210497118546397108136308834859\\
$v_0(24)=$689315097196202964965495137524369491\\
$v_0(25)=$128115739949050222862582840471603630161\\
$v_0(26)=$26485055543003460184635081751737790490077\\
$v_0(27)=$6077461162212148557066452641614974845496911\\
$v_0(28)=$1544989272770276240261220792252418219068759811\\
$v_0(29)=$434325326596750179697353580045791429425537065921\\
$v_0(30)=$134783620687025226578673820825558384519207336154317\\
$v_0(31)=$46097266778852361448504363290889004635531761516739993\\
$v_0(32)=$17347943886981874420443970420868140946401820400065520071\\
$v_0(33)=$7173092361682485947892270771803835972422928260858453607407\\
$v_0(34)=$3254100011661831611232136493298705387412392951373195138695209\\
$v_0(35)=$1617444487417917100687558729024157723212239417105049205352586907\\
$v_0(36)=$879701520010113221282468977866177955184625610469562029928973771369\\
$v_0(37)=$522884716211311457382881933734811248444218233980070070872631230880863\\
$v_0(38)=$339250490647115743863965814841101611982949652546747643275476764917614079\\
$v_0(39)=$239982118166606784159149195685824747348672639149205242576158268123241401993\\
$v_0(40)=$184884449910265737711094690139028338000389030665443236649379240370457305194787\\
$v_0(41)=$1549609747577097322446755568610194649163966726886266135773368469801078703990188\\39\\
$v_0(42)=$1411559151506163897872958974049812681219479080563778825492380720428054694374622\\01529\\
$v_0(43)=$1396049847900249841453889274688081917851190257461611676730098115483833558916985\\42136891\\
$v_0(44)=$1497659982938839633430353086127319678670790952473702024869758096836687311616581\\11306657657\\
$v_0(45)=$1741147329764838562988837412536464007005022149479339551766136907749539386578187\\18931282624893\\
$v_0(46)=$2191698160494759486839021152154495931428729560431669435573190264676516022253722\\16081049741660611\\
$v_0(47)=$2984519935343738066236502899152490668721610060371436230260615961030373758389536\\63122116770277382949\\
$v_0(48)=$4392930391743512399821722042828698182682533737393436918861945128000278551359973\\20619127387923872014543\\
$v_0(49)=$6983436549827930106243584840356879644728504695564535813426135173162536693968717\\91308763924026132123760607\\
$v_0(50)=$1198060655508707745779557753546421413142255593696270754672148422239141624420672\\823287083415254300395417965501\\
\\
The smallest constant of the sequence of primes $v_1(k)=1+\lfloor 2^{k^{d_1}}\rfloor$, $k\geq 0$, is\\
\\
$d_1=$ 1.73551495003300118213908371144667506108878571626358886817760837991871691036752951\\
43276817340763270056256843848592029698421344113178609149345923128149971012097191105368\\
82943842960802987896\\
\\
The first 40 terms of this sequence are\\
$v_1(0)=$2, $v_1(1)=$3, $v_1(2)=$11, $v_1(3)=$107, $v_1(4)=$2179, $v_1(5)=$82567, $v_1(6)=$5583143,\\ 
$v_1(7)=$655201697, $v_1(8)=$130520191667, $v_1(9)=$43329381463117, $v_1(10)=$23597530857908669, $v_1(11)=$20797360003215286823, $v_1(12)=$29308490295076752347591,\\ 
$v_1(13)=$65339739080988479904887093, $v_1(14)=$228234724532861569787299652557,\\
$v_1(15)=$1238274121008683979152081043208907\\
$v_1(16)=$10351929871038747992968349628529933369\\
$v_1(17)=$132377627814659825963442241922323085531281\\
$v_1(18)=$2571924431265830939301333200025838230553258111\\
$v_1(19)=$75444021281138938952784058763440328832674279137379\\
$v_1(20)=$3321792690361474996860493644715305084765269665618308593\\
$v_1(21)=$218334968978888575041251552042761003273126503339991819805711\\
$v_1(22)=$21313024242043949178681237334671968485268350209972689142921272311\\
$v_1(23)=$3074938372000025291498303992795936754870780958689717268965025942217421\\
$v_1(24)=$652699669482172490282676934161157626451563086120385162535269761584611164891\\
$v_1(25)=$202954341407401758711971116110740483720959788623998723081904830450235799301455\\001\\
$v_1(26)=$920682917154801274993099843574038678544045794741683560851671708092876464646214\\93486539\\
$v_1(27)=$606953417470746753972385190764991000233780386339555395497858341488534529044171\\09346308828203\\
$v_1(28)=$579323810655209909407881362707032378086487137190648522567086321337016713792928\\22256881318312054749\\
$v_1(29)=$797754380513626504221502114558810953352344481289765164052602274970557629254761\\23088536682518558068264719\\
$v_1(30)=$157953002384645846684507621385570194843236608277992153896832171414564999397298\\416337569247881342848591283299077\\
$v_1(31)=$448218023429213077510369041645769927017724940200139449677904128693536735557982\\666683613555813590118041533707177284109\\
$v_1(32)=$181720332609031147496763495578115121861473711243821635368727846712411566365782\\2407036267438048265899170654453206580563655933\\
$v_1(33)=$104948088412563777653573720399240836559349403122216127966432308846988750564794\\34676791687814866856932440812157229770756469630996711\\
$v_1(34)=$860907743051222492902299254180584650099908071859078465602195697635458894430580\\81725721622147801810064903826074508934515691184867637172769\\
$v_1(35)=$100034750851738735327051332537110595336215982542251116757087281684550171233925\\
9887158859747928688702884428371720102318582972487111543654892801537\\
$v_1(36)=$164211018565631948395572800562824941849501485439330692989810239903606093555440\\
85079748898408124808926312513804656051993675638200359333137786941844316201\\
$v_1(37)=$379836540175975030113062603753337900591727863301103453960792718994156443595822\\
557063622431071968739751229557701840080004011794990745699224003559273636329651347\\
$v_1(38)=$123497632614194906648728391168283912019616294483661141802611949912339707200507\\
7104374367601242059559681106190766914983600342527565249868001003298345757271014003623\\5767\\
$v_1(39)=$563052020283284540798895877117107285662666099221367919863537142485516428006736\\
7045053306398843609625570038138712518078688899569026372278932734879518502803134389320\\76204179859\\

\section{References}

(1)  Mills, W. H. (1947), A prime-representing function, Bulletin of the American Mathematical Society, 53 (6) p. 604\\
\\
(2)  Polymath (2014), The ``bounded gaps between primes'' Polymath project: a retrospective analysis, Newsletter of the European Mathematical Society, 94 p. 13-23\\
\\
(3)  Cramer, H. (1936), On the order of magnitude of the difference between consecutive prime numbers, Acta Arithmetica, 2 p. 23–46\\
\\
(4)  Shanks, D. (1964), On Maximal Gaps between Successive Primes, Mathematics of Computation, American Mathematical Society, 18 (88) p. 646–651\\
\\
(5)  Granville, A. (1995), Harald Cramer and the distribution of prime numbers, Scandinavian Actuarial Journal, 1 p. 12–28\\
\\
(6)  OEIS, www.oeis.org: A125313\\
\\

\end{document}